\documentclass[letter,reqno, 11pt]{amsart}
\usepackage{graphicx,color}
\usepackage{footmisc}
\usepackage{amsmath}
\usepackage{amsfonts}
\usepackage{amssymb}
\usepackage{bbm}
\usepackage{epstopdf}
\usepackage{color}
\numberwithin{equation}{section}
\usepackage{tikz}
\usepackage{pgfplots}
\usepackage{subfig}
\usepackage{geometry}
\newcommand{\Omm}{\Omega_{L,N;m_1,m_2}}

\newcommand{\OLN}{\Omega_{L,N;n_1,n_2}}
\newtheorem{Theorem}{Theorem}

\newtheorem{Remark}{Remark}

\newtheorem{Lemma}{Lemma}
\newtheorem{Definition}{Definition}

\title[Stationary measure of the $q$-Whittaker particle system]{Stationary measure of the driven two-dimensional $q$-Whittaker particle system on the torus}
\author{Ivan Corwin}
\address{Columbia University,
Department of Mathematics,
2990 Broadway,
New York, NY 10027, USA,
and Clay Mathematics Institute, 10 Memorial Blvd. Suite 902, Providence, RI 02903, USA,
and
Institut Henri Poincar\'e,
11 Rue Pierre et Marie Curie, 75005 Paris, France.}
\email{ivan.corwin@gmail.com}
\author{Fabio Lucio Toninelli}
\address{Universit\'e de Lyon, CNRS and Institut Camille Jordan, Universit\'e Lyon 1, 43 bd du 11 novembre 1918, 69622 Villeurbanne, France}
\email{toninelli@math.univ-lyon1.fr}
\begin{document}

\begin{abstract}
We consider a $q$-deformed version of the uniform Gibbs measure on dimers on the periodized hexagonal lattice (equivalently, on interlacing particle configurations, if vertical dimers are seen as
particles) and show that it is invariant under a certain irreversible
{\it $q$-Whittaker dynamic}. Thereby we provide a new non-trivial
example of driven interacting two-dimensional particle system, or of
$(2+1)$-dimensional stochastic growth model, with
explicit stationary measure. We emphasize that this measure is far
from being a product Bernoulli measure. These Gibbs measures and
dynamics both arose earlier in the theory of Macdonald processes
\cite{BC}. The $q=0$ degeneration of the Gibbs measures reduce to the usual uniform dimer measures with given tilt \cite{Kenyon-Okounkov-Sheffield}, the degeneration of the dynamics originate in the study of Schur processes \cite{BF1,BF2} and the degeneration of the results contained herein were recently treated in \cite{Toninelli}.
\end{abstract}
\maketitle

\section{Introduction}

Irreversible Markovian dynamics on two-dimensional dimers or interlaced particle
configurations (see Figure \ref{fig:corrispondenza}) are closely
related to driven interacting particle systems as well as random
surface growth models in $(2+1)$-dimensions. It is a challenge to find
local irreversible dynamics whose invariant measures (on the torus
or on the infinite lattice) are likewise local and explicit. Knowledge of a dynamic's invariant measures can be useful in establishing its hydrodynamic / fluctuation theory and in understanding general properties (i.e. universality classes) of two-dimensional driven systems.

The state space for the Gibbs measures / dynamics we consider is that of interlacing particles on a discrete torus. Each particle interlaces with two particles above it, two particle below it, and has two neighbors at the same row (see Figure \ref{fig:particelle}). For a particle $p$ we let $A_p, B_p, C_p, D_p, E_p, F_p$ denote the absolute value of the horizontal distances between the particle and its neighbors starting with the right neighbor on the same row and going clockwise (actually, the values of $A_p, B_p, D_p, E_p$ are this distance minus 1 -- see Section \ref{sec:particelle} for precise definitions). In continuous time, each particle $p$ jumps to the right by one lattice space according to an exponential clock of rate
$$
\frac{(1-q^{B_p})(1-q^{D_p+1})}{(1-q^{C_p+1})}
$$
where $q\in [0,1)$. This dynamic is local and irreversible. Moreover, it preserves the interlacing of particles -- particle $p$ cannot jump past the particle below and to its right since the factor $1-q^{B_p}=0$ in that scenario; and if a particle $p$ jumps so as to pass its upper-right neighbor, then that neighbor is immediately pushed right by one space since the denominator for that neighbor's jump rate is zero, hence it jumps right at an infinite rate.

The main result of this paper, Theorem \ref{th:invarianza}, provides
an explicit two-parameter family of translation invariant Gibbs measures on the torus which are invariant for this dynamic. One parameter represents the number of particles per row, and the other is topological, related to how much the particles rotate right between rows (in terms of dimers, the parameters relate to the slopes of the height function). The Gibbs measures are local and proportional to
$$
\prod_p\frac{(q;q)_{A_p}}{(q;q)_{B_p}(q;q)_{C_p}}
$$
where the product is over all particles $p$ and where $(q;q)_n =
(1-q)(1-q^2)\cdots (1-q^n)$ is the $q$-Pochhammer symbol. These
dynamics are irreversible (i.e. driven); the proof of the invariance
of the Gibbs measure involves unexpected cancelations of contributions
from every particle. For $q=0$,  the
dynamic reduces to the one introduced by A. Borodin and P. L. Ferrari
\cite{BF1} and the Gibbs measures reduce to uniform
measures over the configurations of interlacing particles. Their stationarity on the torus was proved in \cite{Toninelli} and the argument is simpler than in the $q\in(0,1)$ case.

In general, there is no hope of computing explicitly the stationary
measures of a driven particle system. In the
lucky cases when it is possible, often the  invariant measures turn out to be of
product type. This is the case, for instance, for the ASEP on
$\mathbb{Z}^d,d\geq 1$ and certain modifications thereof
\cite{QuastelValko}, zero range processes \cite{Spitzer, Andjel}, mass
transport models \cite{EvansMajumdarZai}, and certain
discretizations of the one-dimensional KPZ equation including directed polymer models
\cite{SasamotoSpohnSuper,SeppLogGamma,OConnellYor,BetaPolymer}. In
contrast, our stationary measures $\pi$ are far from being product:
for instance, for $q=0$ it is known that they show power-law decaying spatial
correlations and the same presumably holds for $q\ne0$.

The dynamics and Gibbs measures we study in this work were initially considered in the study of Macdonald processes (in fact, $q$-Whittaker processes) \cite[Chapter 3]{BC}, though our results do not rely at all on this technology. Still, let us briefly explain the origins. The $q$-Whittaker process is a measure on interlacing triangular arrays with $1$ particle on the first (bottom) row, $2$ on the second, up to $N$ on the top ($N^{th}$) row. Given the configuration of particles on the top row, the $q$-Whittaker process enjoys the property that the measure on the remaining $N-1$ lower rows is given by the above described Gibbs measure. The dynamics we described were also introduced in \cite[Chapter 3]{BC} and play nice with such measures on triangular arrays. In particular, if one starts with an interlacing triangular array satisfying the Gibbs property, then after running the dynamics for a fixed time, the resulting triangular array also enjoys the same Gibbs property (of course, the top row will have moved to the right). This fact follows from results in \cite[Section 2.3]{BC}.

While our main result does not follow from the results of \cite{BC}, it was partially inspired by it. After all, if the Gibbs property is preserved on these triangular arrays, it would seem reasonable that it should also be preserved on the periodized version of the model. In \cite{BC} a more general class of dynamics and Gibbs measures are discussed which are inhomogeneous with row-dependent parameters $a_k>0$. Our approach applies equally well in that more general context and so we state and prove our main theorem for general $a_k$ (it reduces to the case described above when all $a_k\equiv 1$). There actually exist many types of continuous time Markov dynamics on triangular arrays which preserve the Gibbs property \cite{BorPet}, however the other known dynamics do not easily admit periodic versions (i.e. they are not translation invariant). There is also a discrete time version of the above introduced dynamic described in \cite{BC}, however it is not as simple and we do not pursue studying its periodized version.

The other inspiration for our result is the recent work \cite{Toninelli} regarding the $q=0$ case of this model  (that work was inspired by the Schur process work of
\cite{BF1,BF2}). The $q=0$ case of the dynamics is known to lie in the $(2+1)$-dimensional anisotropic Kardar-Parisi-Zhang universality class \cite{BF1,BF2,Toninelli}.
Based on the approach developed in \cite{Toninelli}, the results of the present paper can be seen as a step towards extending this universality to $q\in (0,1)$. In the $q=0$ case, the next step in \cite{Toninelli} is to extend the invariance of Gibbs measures from the finite torus to infinite volume. A crucial ingredient used in this extension was the fact that, for
$q=0$, the infinite-volume Gibbs states are known and have an explicit determinantal structure and GFF-like height fluctuations \cite{Kenyon-Okounkov-Sheffield}. All of this is missing in the $q\in
(0,1)$ case, so at present, the extension of our result to the infinite lattice and the proof of its KPZ class behavior is an open problem.

Recently, Borodin-Bufetov \cite{BorodinBufetov}
considered another deformation of the Borodin-Ferrari particle system ($q=0$ case of our dynamics) and they
showed that the invariant measures are given by 6-vertex Gibbs
measures. This deformation is different from the one we consider here (ours originates from $q$-Whittaker processes \cite{BC} and the other from vertex models \cite{BorodinR}). Let us also mention that another example of $(2+1)$-dimensional growth
model with explicit (non-product form) stationary measure is the Gates-Westcott model \cite{GatesWestcott, PrahoferSpohn}.

\subsubsection*{Outline} In Section \ref{sec.model} we introduce the
state space for our Markov dynamics (in terms of dimers as well as
interlacing particle configurations) and describe the class of Gibbs
measures we will work with. Section \ref{sec.dynamics} introduces the periodized $q$-Whittaker dynamics. Section \ref{sec.invariance} contains the proof that the Gibbs measures are invariant for the $q$-Whittaker dynamics.

\section{State space and Gibbs measure}\label{sec.model}

The dynamics we study can be described as a driven interacting particle system, or as dynamics of a dimer model on the (periodized) hexagonal lattice. While the former description may be more natural, the latter allows to get more directly certain statements, notably ergodicity of the Markov chain. We start by introducing the state spaces for our dynamics.

\subsection{Description as an interacting particle system}
\label{sec:particelle}
The particle process lives on the $L\times N$ discrete torus $\mathbb T_{L,N}=\mathbb Z/(L \mathbb Z)\times \mathbb Z/(N \mathbb Z)$. The horizontal size is $L$ and the vertical size is $N$. % Given
% two sites $x,y$ on $\mathbb
% Z/(L\mathbb Z)$, let $I(x,y)=\{x+1,\dots,y\}$.  For instance,
% $I_{L-1,2}=\{L,L,1,2\}$.

The particle configuration space will be denoted $\Omm$, and depends on two integers $1< m_1< L$ and $1\le m_2< N$ such that
\begin{gather}
  \label{eq:26}
m_1/L+m_2/N<1.
\end{gather}
 At each site $x=(x_1,x_2)\in \mathbb T_{L,N}$ there is at most one particle. On each row there are exactly $m_1$ particles. We exclude $m_1=1$ and $m_1=L$ to avoid trivialities. The
 parameter $m_2$ has a more topological nature and its meaning will be explained below.

Particle positions are interlaced, in the following sense. The horizontal position of particle $p$ is denoted  $x_p\in \mathbb Z/(L\mathbb Z)$. Given any $p$ (say on row $i$), we let $p_1,p_4$ denote its right/left neighbor on the same row (note that if $m_1=2$ then $p_1=p_4$). Then, we require that in row $i-1$ there is exactly one particle, denoted $p_2$, whose position satisfies
%conditions: for $2\le i\le N$, $0\le j\le m -1$
\begin{eqnarray}\label{eq:1}
x_{p_2}\in \{x_{p}+1,x_p+2,\dots,x_{p_1}\}
\end{eqnarray}
and exactly one particle, denoted $p_3$, satisfying
\begin{gather}
  \label{eq:22}
  x_{p_3}\in \{x_{p_4}+1,x_{p_4}+2,\dots,x_{p}\}.
\end{gather}
See Figure \ref{fig:particelle}. Note that, automatically, in row $i+1$ there are exactly one particle $p_5$ and one
particle $p_6$ satisfying respectively
\begin{gather}
  \label{eq:23}
  x_{p_5}\in \{x_{p_4},\dots,x_p-1\},\quad
x_{p_6}\in \{x_{p},\dots,x_{p_1}-1\}.
\end{gather}

\begin{figure}[ht]
%\centering
%\resizebox{.3\textwidth}{!}{
\includegraphics[width=.6\textwidth]{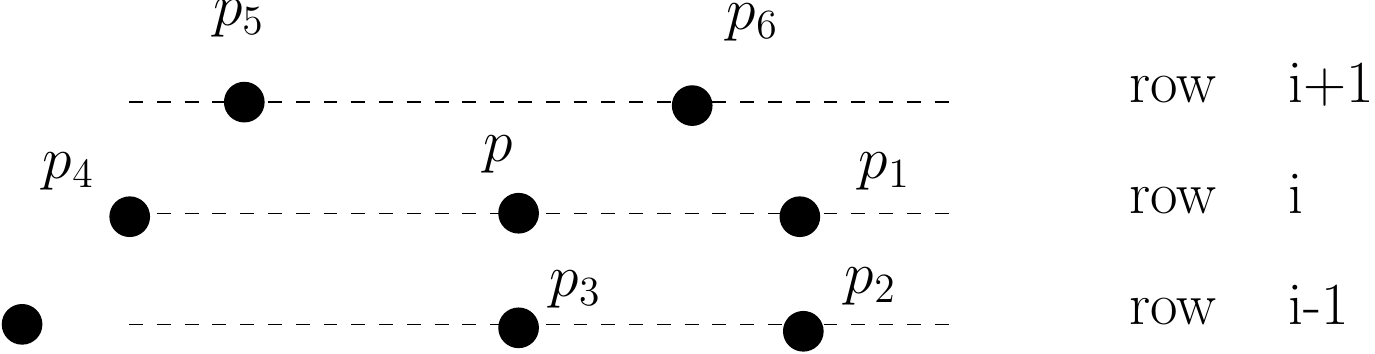}
%\includegraphics[width=0.48\textwidth]{feyn1.pdf}}
%}
\caption{
The neighbors $p_1,\dots,p_6$ of particle $p$. Note that conditions
\eqref{eq:1}, \eqref{eq:22} allow $C_p:=x_{p}-x_{p_3}=0$ but they impose
$B_p+1:=x_{p_2}- x_p\ge1$.
}
\label{fig:particelle}
\end{figure}
We define non-negative integers $A_p,\dots,F_p$ as
\begin{eqnarray}
  \label{eq:19}
  &A_p=x_{p_1}-x_{p}-1;\quad
B_p=x_{p_2}-x_p-1;\quad
C_p=x_p-x_{p_3}\\
\nonumber
&D_p=x_{p}-x_{p_4}-1;\quad
E_p=x_p-x_{p_5}-1;\quad
F_p=x_{p_6}-x_p.
\end{eqnarray}
The particles $p_1,\dots,p_6$ are the six neighbors of $p$, labeled clockwise starting from the one on the right. The definition of the dynamics will be such that the labels of the neighbors of a particle $p$ do not change with time.

Let $\Omega_{L,N;m_1}$ be the set of particle occupation functions, i.e. of functions $\eta:\mathbb T_{L,N}\mapsto \{0,1\}$, with $m_1$ particles per row, whose positions satisfy the constraints \eqref{eq:1}-\eqref{eq:23}. The set $\Omega_{L,N;m_1}$ decomposes into disjoint ``sectors'':
\begin{gather}
  \label{eq:24}
  \Omega_{L,N;m_1}=\cup_{m_2}\Omm
\end{gather}
as follows. Given any particle $p$, connect $p$ to its up-right neighbor $p_6$, then $p_6$ with its own up-right neighbor and repeat the operation until the path $\Gamma$ thus obtained gets back to the starting particle $p$. Note that $\Gamma$ forms a simple loop: otherwise, there would be a particle $r$ which is reached along $\Gamma$ from two different particles $r',r''$. This is impossible, since both $r'$ and $r''$ would be the $r_3$ neighbor of $r$. Call $N_v\in\mathbb N\cup\{0\},N_h\in\mathbb N\cup\{0\}$ the vertical and horizontal winding numbers  of $\Gamma$ around the torus $\mathbb  T_{L,N}$. It is easy to see that $N_h,N_v$ are independent of the chosen initial particle $p$.
 % Note that
% $p'$ need not be $p$ itself: we say that the particle configuration belongs to
% the sector $\Omm$ if $p'$ is the $M$-th left neighbor of $p$ (if
% $M=0$, it means that $p'=p$).
It is possible to check (see Remark \ref{rem:km} below) that
\begin{gather}
  \label{eq:27}
  m_2:=m_1\frac{N_h}{N_v}
\end{gather}
is an integer and that it satisfies \eqref{eq:26}. The set $\Omega_{L,N;m_1,k}$ is the subset of $\Omega_{L,N;m_1}$ such that $ m_2=k$. Each sector $\Omm$ will remain invariant under our dynamics.

% Each
% particle $p$ is indexed by $(i,j)$, with $1\le i\le N$ the row
% (counted starting from
% the lowest one) and $0\le j\le m-1$ the particle label in row $i$.
% Whenever in the formulas below $j$ is not in $\{0,\dots,n-1\}$, we mean that $(i,j)$ is the
% particle $(i,\mod(j,n))$.
% In each row, particles are labeled in increasing order: $(i,1)$
% is to
% the right of $(i,0)$, etc., $(i,0)$ is to the right of
% $(i,n-1)$. (Recall that rows are periodized).

% Then, particle positions are required to satisfy the following

% while for $i=1$
% \begin{eqnarray}
%   \label{eq:2}
%   x_{(1,j)}\in I(x_{(N,j+M)},x_{(N,j+M+1)}).
% \end{eqnarray}
% We let $\Omm$ to be the set of particle positions
% satisfying \eqref{eq:1}-\eqref{eq:2}.

Given $q\in [0,1)$ and a collection of real parameters $a_1,\ldots, a_N>0$, let $\pi$ be the probability measure on
$\Omm$ defined as
\begin{eqnarray}
  \label{eq:3}
  \pi(\sigma):=\pi_{L,N;m_1,m_2}(\sigma):=\frac1{Z_{L,N;m_1,m_2}}\prod_p
  a_{r(p)}^{C_p} \frac{(q;q)_{A_p}}{(q;q)_{B_p}(q;q)_{C_p}}{\bf 1}_{\{\sigma\in \Omm\}}
\end{eqnarray}
where $(q;q)_n=(1-q)(1-q^2)\dots (1-q^n)$ and where $r(p)$ denotes the label of the row in which $p$ sits.

Conditionally on the position of all particles except particle $p$, the law of the position of $p$ is proportional to
\begin{eqnarray}
  \label{eq:5}
  a_{r(p)}^{C_p}a_{r(p_6)}^{F_p}\, \frac{(q;q)_{A_p}(q;q)_{D_p}}{(q;q)_{B_p}(q;q)_{C_p}(q;q)_{E_p}(q;q)_{F_p}}.
\end{eqnarray}
Note also that in \eqref{eq:3} we could have for instance replaced
$(q;q)_{B_p}(q;q)_{C_p}$ by $(q;q)_{E_p}(q;q)_{F_p}$ and $\pi$ would
be unchanged. Likewise, because $\sum_{p:r(p)=k} C_p+B_p = N-m_1$ we
could have replaced $a_{r(p)}^{C_p}$ by
$a_{r(p)}^{\alpha C_p- (1-\alpha) B_p}$ for any $\alpha$ and not
changed $\pi$ either.  When $q=0$ and all $a_k\equiv 1$, $\pi$ reduces
to the uniform distribution on $\Omm$.

\begin{Remark}
Note that we are viewing $\Omm$ as a set of particle occupation functions and not as a set of positions of labeled particles; i.e. two particle configurations with the same particle occupation
variable everywhere are considered to be the same. If instead we looked at configurations of labeled particles, every configuration in $\Omm$ would actually correspond to $m_1$ different possible particle configurations, obtained by cyclically changing particle labels, simultaneously on every row.
\end{Remark}
% and let $\sigma$ denote the
%generic configuration in $\Omega_{L,N;M,n}$.

%  it is composed by  $N$
% rows $C_i,1\le i\le N$ with $L$ positions each:
% $C_i=\{1,\dots,L\}$. Positions on the torus are indexed $x=(x_1,x_2)$ with
% $1\le x_1\le L$ and $1\le x_2\le N$.  On a row, say that clockwise direction means increasing coordinate.

%FIGURE WITH $n=3, N=3, L=something$.
% The $j^{th}$ particle in the $i^{th}$ circle is labeled $p=(i,j)$,
% $1\le i\le N, 0\le j\le n-1$. On
% each circle, they are placed in clockwise increasing order
% (i.e. etc).

% The

\subsection{Description as a dimer model}
Before we introduce the dynamics, let us give the alternative description of the model in terms of perfect matchings (or dimer coverings) of the periodized hexagonal lattice $\Lambda_{L,N}$.
Precisely, with reference to Figure \ref{fig:grafo},  $\Lambda_{L,N}$ has period $L$ in the $e_1$ direction and period $N$ in the $e_2$ direction, and vertices are alternately colored black/white.

\begin{figure}[ht]
%\centering
%\resizebox{.3\textwidth}{!}{
\includegraphics[width=.7\textwidth]{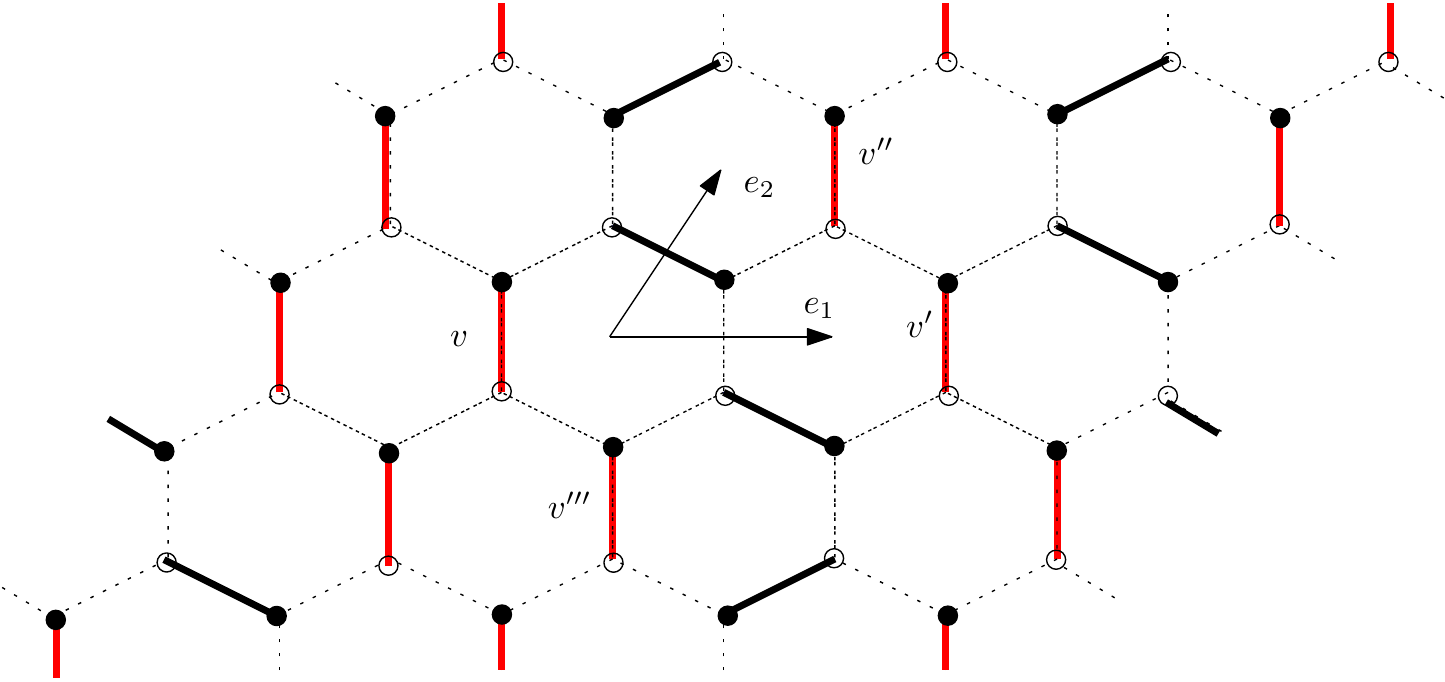}
%\includegraphics[width=0.48\textwidth]{feyn1.pdf}}
%}
\caption{
The graph $\Lambda_{L,N}$ with $L=5,N=4$. The graph is periodic in both $e_1$ and $e_2$ directions. Vertical dimers (in red) will be called ``particles''. Here, $v'=v+2 e_1, v''=v+e_2+e_1$ and $v'''=v-e_2+e_1$.
}
\label{fig:grafo}
\end{figure}
The number of dimers (i.e. of edges in the perfect matchings) equals $LN$, which is the number of white vertices. Given strictly positive integers $n_1,n_2 $ with $n_1+n_2<N L,$ let $\OLN$ be the set of dimer coverings $X$ with $n_1$ vertical dimers and $n_2$ north-west oriented dimers. It is well known that the number of vertical dimers is the same in each horizontal row, so that $k_1:=n_1/N\in \mathbb N$, and similarly there are $k_2:=n_2/L\in \mathbb N$ north-west oriented dimers in each $e_2$-oriented column. To avoid trivialities we will also assume that $k_1\ge 2$ (recall that in the particle picture we required $m_1\ge2$ for the same reason). Another well-known fact is that the knowledge of which vertical edges are occupied uniquely determine the whole dimer configuration. A last observation is that the horizontal (i.e. $e_1$) coordinates of vertical dimers are interlaced: given two vertical edges $v$ and $v'=v+k e_1,k>0$ such that there is a dimer at $v,v'$ and no dimer  at $v+je_1,
j=1,\dots,k-1$, then there is exactly one value $0\le r<k$ and one value $0<s\le k$ such that there is a dimer at $v''=v+e_2+r e_1$ and at $v'''=v-e_2+s e_1$. See Figure \ref{fig:grafo}.

\begin{Remark}
\label{rem:km}
If vertical dimers are called ``particles'', then the correspondence between the particle picture of Section \label{sec:part} and the dimer picture should now be clear, see also Figure  \ref{fig:corrispondenza}: for each particle configuration in $\Omm$ there is a perfect matching of $\Lambda_{L,N}$ with certain values $(k_1,k_2)$, and vice-versa. It is obvious that $m_1=k_1$.
It remains to be shown that also $m_2=k_2$, thereby proving also \eqref{eq:26}, since $L k_2+N k_1=n_1+n_2<NL$ by assumption. Given a nearest-neighbor path $C_{f\to f'}$ from a face $f$ to a face $f'$ and a perfect matching $X$, define
\begin{gather}
  \label{eq:25}
  H_X(C_{f\to f'})=\sum_{e\in C_{f\to f'}}\epsilon_e {\bf 1}_{e\in X}
\end{gather}
where the sum runs over the edges crossed by the path, $\epsilon_e$ equals $+1/-1$ if $e$ is crossed with the white vertex on the right/left and ${\bf 1}_{e\in X}$ is the indicator function that $e$ is occupied by a dimer in $X$. It is well known \cite{Kenyon} that if $f'=f$ then, given $X$, $H_X(C_{f\to f})$ depends only on the horizontal and vertical winding numbers of $C_{f\to f}$ around the torus. Choose a path $C^{(1)}_{f\to f}$ that moves only in the $+e_1$ direction. Then, we see that
\begin{gather}
  \label{eq:28}
  H_X(C^{(1)}_{f\to f})= +k_1.
\end{gather}
If instead $C^{(2)}_{f\to f}$ it moves only in the $+e_2$ direction, then
\begin{gather}
  \label{eq:29}
 H_X(C^{(2)}_{f\to f})=-k_2.
\end{gather}
On the other hand, choose $C^{(\Gamma)}_{f\to f}$ as follows: it
starts from a face $f$ just to the right of a vertical dimer $p$, it
moves in direction $+e_2$ if this involves crossing no dimer, and in
the direction $+e_1$ otherwise; the paths stops when it gets back to
$f$. Let us say that a vertical dimer $r $ ``is visited by
$C^{(\Gamma)}_{f\to f}$'' if the path visits the hexagonal face just
to the right of $r$. Note that the first visited vertical dimer is
$p$, the second one is its up-right neighbor $p_6$, then the up-right
neighbor of $p_6$, and so on. These are just the particles visited by
the  path $\Gamma$ built in Section \ref{sec:particelle}. In
particular, we have seen that $\Gamma$ forms a simple loop, so
$C^{(\Gamma)}_{f\to f}$ does indeed come back to $f$. Also, $C^{(\Gamma)}_{f\to f}$  has the same winding numbers $N_h,N_v$ as $\Gamma$. By construction, we see that $H_X(C^{(\Gamma)}_{f\to f})=0$ since the path never crosses any dimer. On the other hand, by \eqref{eq:28}-\eqref{eq:29} we have
\begin{gather}
  \label{eq:30}
 0= H_X(C^{(\Gamma)}_{f\to f})=N_h k_1-N_v k_2,
\end{gather}
i.e. $k_2=m_2$, see \eqref{eq:27}.
\end{Remark}

\begin{figure}[ht]
%\centering
%\resizebox{.3\textwidth}{!}{
\includegraphics[width=.32\textwidth]{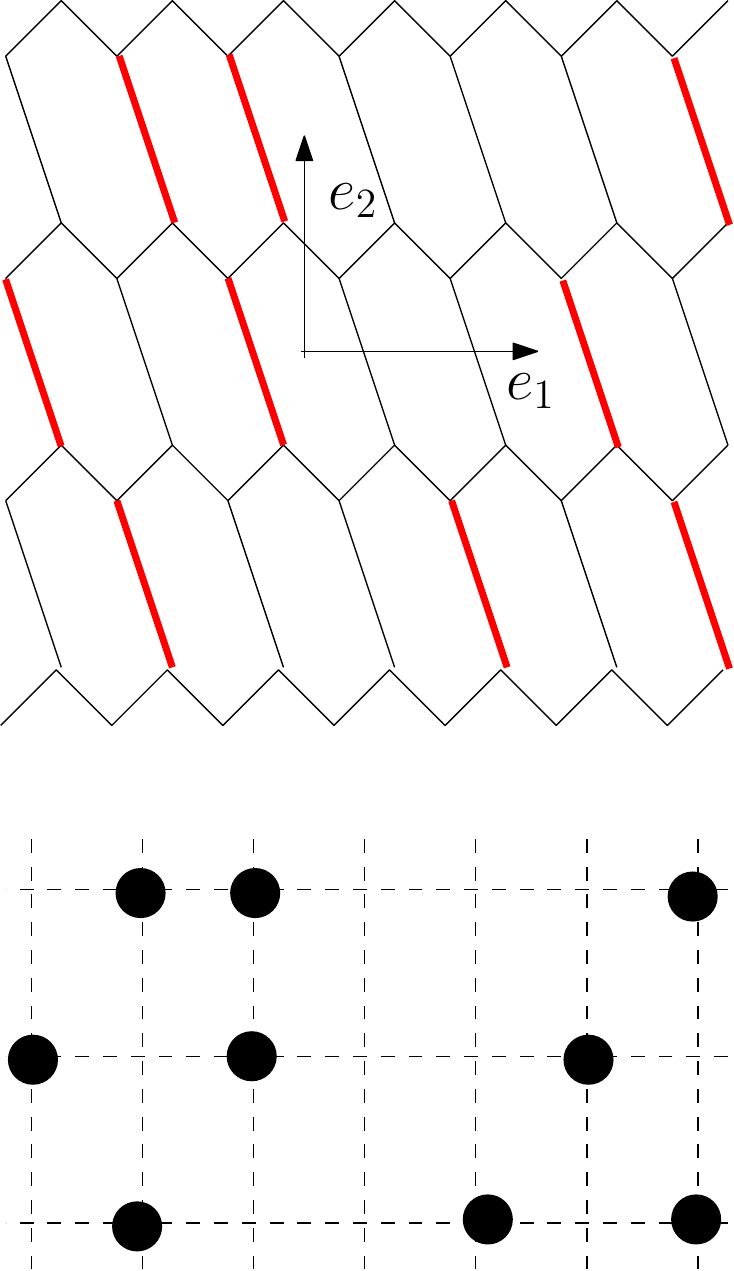}
%\includegraphics[width=0.48\textwidth]{feyn1.pdf}}
%}
\caption{
The correspondence between (a portion of) perfect matching and interlacing particle configuration. The mapping is made more evident by an affine deformation of the hexagonal lattice such that the axes
$e_1,e_2 $ become orthogonal}
\label{fig:corrispondenza}
\end{figure}

In analogy with Section \ref{sec:particelle}, for each vertical dimer $p$, let $p_1, \dots p_6$ be the labels of its six neighboring vertical dimers labeled  in clockwise order starting from the one on the right. The non-negative (possibly zero) integers $A_p,$ $B_p,C_p,D_p,E_p,F_p$ defined in \eqref{eq:19} are given in the coordinates of the hexagonal lattice as follows: if $p$ is at vertical edge $v$, then
\begin{align*}
&p_1  \text{ is at edge } v+(A_p+1) e_1; &&p_2 \text{ is at edge } v-e_2+ (B_p+1)e_1\\
&p_3 \text{ is at edge } v-e_2-C_p e_1; &&p_4 \text{ is at edge } v- (D_p+1) e_1\\
&p_5 \text{ is at edge }v+e_2-(E_p+1) e_1; &&p_6 \text{ is at edge } v+e_2+F_pe_1.
\end{align*}

\section{Periodized $q$-Whittaker dynamics}\label{sec.dynamics}
We saw that the descriptions in terms of interlaced particles on the torus $\mathbb T_{L,N}$ or of perfect matchings of $\Lambda_{L,N}$ are equivalent, if we set $(m_1,m_2)=(k_1,k_2)$. We will interchangeably use the former or the latter representation.
\begin{Definition}
\label{def:famiglie}
Given a configuration $\eta\in \Omm$, draw a directed upward edge from any particle $r$ to its up-right neighbor $r_6$ if $F_r=0$ and a downward edge from $r$ to $r_3$ if $C_r=0$ (in both cases, we draw an edge between particles in neighboring lines, with the same horizontal position). For each particle $p$ let  $V^+_p$ (resp. $V^-_p$) be the set that includes $p$ plus the particles that can be
reached from $p$ by following upward (resp. downward) oriented edges.
% particles splits into connected families $C_i$. In any family, the
% lowest particle is the one that needs to push everyone else to move,
% the highest is the one that does not need to push anyone.
\end{Definition}

\begin{Remark}
From the assumption $n_1,n_2>0$ it follows that $|V^\pm_p|\le N-1$ for every $p$. In fact, if say $|V^+_p|\ge N$ then the upward edges starting from $p$ form a loop and actually $|V^+_p|=N$ (in each row there is at most one particle $p'$ with the same horizontal position $x_{p'}$ as $p$). In this case, the path $\Gamma$ defined just after \eqref{eq:24} has zero horizontal winding number, so from  \eqref{eq:27} we have $m_2=n_2=0$. Therefore, the up/down arrows do not form loops and we can identify $r^+_p$, the highest particle in $V^+_p$, and  $r^-_p$, the lowest particle in $V^-_p$.
\end{Remark}

The dynamics we consider is a continuous-time Markov chain on $\Omm$. The updates consist in shifting by $+e_1$ all particles in one of the families $V^+_p$. Such move happens with rate
\begin{eqnarray}
  \label{eq:6}
  a_{r(p)} \, \frac{(1-q^{B_p})(1-q^{D_p+1})}{(1-q^{C_p+1})},
\end{eqnarray}
where $q\in [0,1)$ and $a_1,\ldots, a_N>0$ are the same real parameters as in the definition of the Gibbs measure \eqref{eq:3}. Recall that $r(p)$ is the row associated to particle $p$.
Note that the rate is zero if $B_p=0$. This prevents particles from overlapping after the move, see Figure \ref{fig:ancoraO} (b). Note also that after the move, the configuration is still in
$\OLN$. Indeed, if $|V_p^+|=1$, shifting $p$ by $+e_1$ corresponds to rotating three dimers around a hexagonal face (see Figure \ref{fig:ancoraO} (a)), so $n_1,n_2$ remain constant. In the general case, if $B_p>0$ the shift of $V_p^+$ by $e_1$ can be obtained by shifting $+e_1$ the particles of $V_p^+$ one by one, starting from the top one, i.e. $r_p^+$. Also in this case, $n_1,n_2$ are unchanged.

\begin{figure}[ht]
%\centering
%\resizebox{.3\textwidth}{!}{
\includegraphics[width=.7\textwidth]{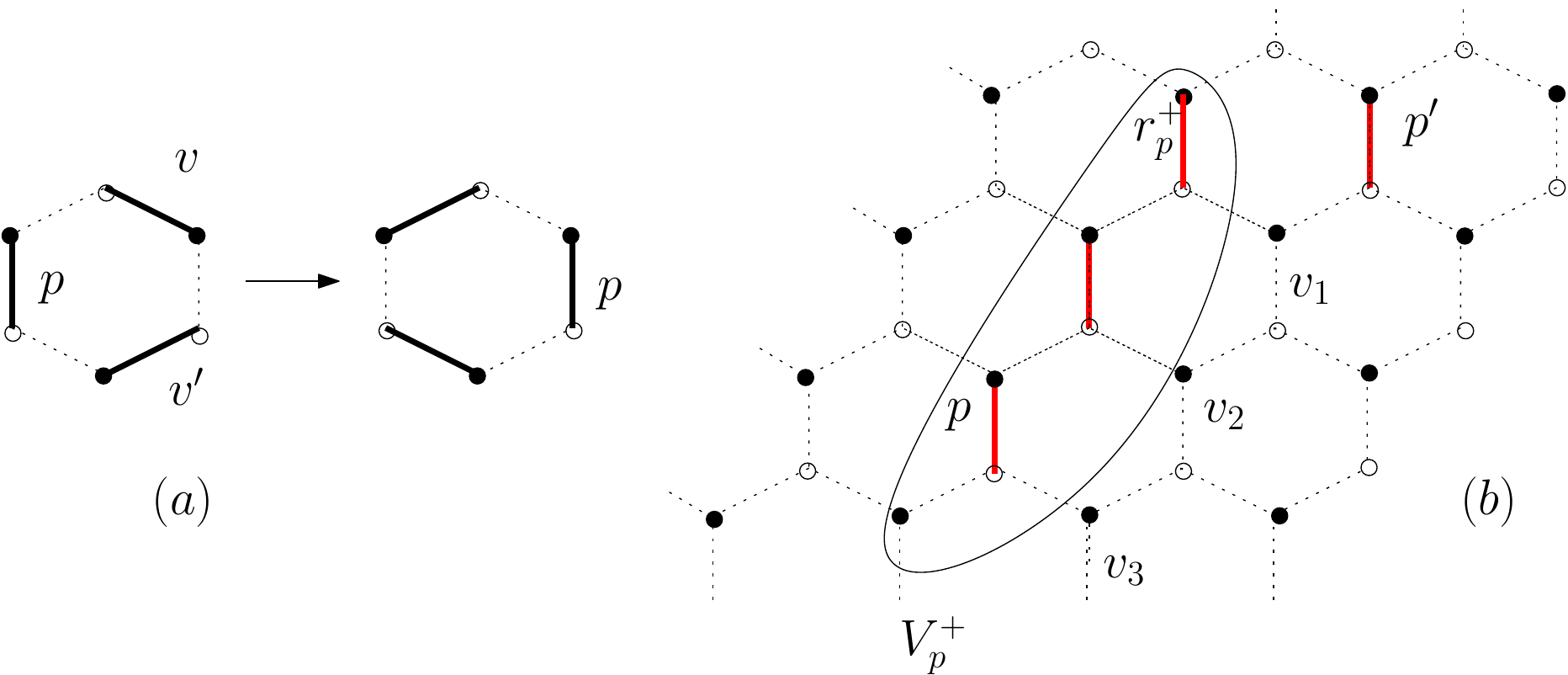}
%\includegraphics[width=0.48\textwidth]{feyn1.pdf}}
%}
\caption{(a): If $|V_p|=1$ then necessarily edge $v$ is occupied and if $B_p\ne 0$ then $v'$ is also occupied. The allowed shift of $e_1$ then corresponds to the rotation of these three dimers.
(b): A particle $p$ with the associated set $V^+_p$, whose highest particle is $r^+_p$. The shift of  $V^+_p$  by $e_1$ should be forbidden, otherwise particles $r^+_p$ and $p'$ would overlap. Indeed, by the particle interlacing condition note that edges $v_1,v_2,v_3$ are necessarily occupied by particles. In particular, $B_p=0$  because $v_3$ is occupied, so that the rate \eqref{eq:6} is zero.}
\label{fig:ancoraO}
\end{figure}

\begin{Remark}
For $q=0$ and $a_k\equiv 1$, the dynamics was studied in
\cite{BF1,BF2,Toninelli} (in this case, $q^{B_p}$ should be
interpreted as $1$ when $B_p=0$). As we already remarked, in this case $\pi=\pi_{L,N;m_1,m_2}$ is the uniform measure over $\Omm$. In \cite{Toninelli}, ``particles'' were associated to north-west oriented dimers, rather than to vertical ones as in the present work. With the convention of \cite{Toninelli}, updates consist in a single particle jumping a distance $n\ge1$ in the $-e_2$ direction, instead of $n\ge 1$ particles jumping a distance $1$ in the $+e_1$ direction.
\end{Remark}

A preliminary observation:
\begin{Lemma}
\label{lemma1}
  The Markov
chain is ergodic on $\OLN$. More precisely, one can go from any
$\eta \in \OLN$ to any $\eta'\in \OLN$ via a chain of elementary moves
where a single particle jumps by $ +e_1$.
\end{Lemma}
\begin{proof}[Proof of Lemma \ref{lemma1}]
First let us define the height function $h_{\eta,\eta'}$ of $\eta$ respective to $\eta'$ \cite{Kenyon}: this is defined on hexagonal faces $f$ and its gradients are given by
\begin{align}
  \label{eq:21}
  h_{\eta,\eta'}(f')-  h_{\eta,\eta'}(f)&=H_\eta( C_{f\to
  f'})-H_{\eta'}( C_{f\to
  f'})\\
  &\nonumber =
\sum_{v\in C_{f\to
  f'}}\epsilon_v ({\bf 1}_{v\in \eta}-{\bf 1}_{v\in\eta'})
\end{align}
where $H_\eta( C_{f\to
  f'})$ was defined in \eqref{eq:25}.

The r.h.s. of \eqref{eq:21} is independent of $C_{f\to f'}$: this amounts to the statement that $  h_{\eta,\eta'}(f')- h_{\eta,\eta'}(f)=0$ if $f'=f$, which follows from the second equality in \eqref{eq:30} because both $\eta,\eta'$ are in $\OLN$ with the same values of $n_1=N k_1,n_2=L k_2$. The height function is therefore well-defined modulo an additive constant: let us fix this constant by establishing that $h_{\eta,\eta'}\ge 0$ and that it vanishes at some face that we call $\bar f$. Given any $f$ such that $h_{\eta,\eta'}(f)>0$, construct a path $f=f_0, f_1, f_2,\dots$ with the rule that the edge traversed from $f_i$ to $f_{i+1}$ is \emph{not} occupied in $\eta$ and has the white vertex on the left. We see then that $h_{\eta,\eta'}$ is non-decreasing along this path.
We have:
\begin{Lemma}
\label{lemma:loop}
The path $f=f_0, f_1, f_2,\dots$ cannot contain a closed loop.
\end{Lemma}
As a consequence, it must be the case that after a finite number $k$ of steps the path cannot be continued. In this case, this implies that at face $f_k$ it is possible to rotate the three dimers (all three edges with white vertex on the left are occupied, as in Figure \ref{fig:ancoraO} (a)): the vertical one moves by $e_1$ and the height at $f_k$ decreases by $1$. Repeat the procedure until $h_{\eta,\eta'}$ is zero everywhere.

% The proof of ergodicity works in two steps:
% \begin{itemize}
% \item first we find a chain of elementary moves (a particle moving
%   $e_1$) at the end of which the height function is zero everywhere.
% \item at this point, the configuration obtained has the same particle
%   occupation variable at every site.
% \end{itemize}
\end{proof}

\begin{proof}[Proof of Lemma \ref{lemma:loop}]
Assume by contradiction that the path contains a loop $C_{g\to g}$. We know that for every $X\in \OLN$ we have
\begin{gather}
  \label{eq:31}
  H_X(C_{g\to g})=H_\eta(C_{g\to g}).
\end{gather}
Since all edges traversed by $C_{g\to g}$ have the white site on the left, $ H_X(C_{g\to g})$ equals minus the number of traversed dimers. For $\eta$, we know that no dimer is traversed, so we conclude that none of the dimers traversed by $C_{g\to g}$ is occupied in any configuration $X\in \OLN$. This leads to a contradiction with our assumption $n_1,n_2,NL-n_1-n_2>0$. Suppose for instance that a vertical edge is contained in no  $X\in \OLN$: by translation invariance, this would imply $n_1=0$. Similarly, the assumption $n_2>0$ and $LN-n_1+n_2>0$ exclude the case where some
given non-horizontal edge is contained in none of the configurations $X\in \OLN$.
\end{proof}

% *****

% In particular, the rate is zero
% if $B^+_p$ includes $p_1$ (which is the case if $|B^+_p|>N$).
% THIS SENTENCE MAYBE IS NOT VERY CORRECT.
% ****

% After the update, it is possible that $f_p<0$ (particle $p$ has
% overcome its up-right neighbor $p'$). In this case, we also move
% instantaneously $p'$  to the right. If the up-right
% neighbor $p''$
% of $p'$ has been overcome by $p'$, then we move $p''$ to the right and
% so on. In one update, there are I would say $K\le N-1$ particles that
% move by $+1$ to the right (one is the particle that jumps, the other
% $K-1$ are the ones who are pushed by the first). Note that a particle
% $p$ cannot push the one to its right on the same row: indeed, if this
% neighbor is at the same position as $p$, then necessarily $b_p=0$ and
% the jump rate is zero.

\section{Invariance of Gibbs measures}\label{sec.invariance}

Our main result is:
\begin{Theorem}
\label{th:invarianza}
The probability law $\pi:=\pi_{L,N;m_1,m_2}$ is stationary in time.
\end{Theorem}

\begin{Remark}
In \cite{Toninelli}, stationarity  of $\pi$ for $q=0$ and all $a_k\equiv 1$ was proven, and actually it was shown that the infinite volume Gibbs measures $\pi_{\rho_1,\rho_2}$, obtained in the limit $L\to\infty,N\to\infty$, $m_1/L\to\rho_1, m_2/N\to \rho_2$, are stationary for the dynamics on the infinite hexagonal lattice. The proof of invariance of $\pi$ on $\Lambda_{L,N}$ in the case $q>0$ is more involved than for $q=0$.
\end{Remark}

%\begin{Remark}
%\red{I would remove this remark, this generalization does not sound so
%  exciting.
%Also, I am not sure about reflections. for instance when you reflect
%through the horizontal axis, north-east edges become nort-west. Why
%should the measure be invariant?}
%The Gibbs measure $\pi$ is invariant under reflection through the horizontal and vertical axis. Therefore, Theorem \ref{th:invarianza} immediately implies that such vertical or horizontal reflections of the $q$-Whittaker dynamics also preserve $\pi$ (and likewise any linear combination of their generators). This produces a larger family of asymmetric local driven interacting particle systems with the same explicit invariant measure.
%\end{Remark}

\begin{proof}[Proof of Theorem \ref{th:invarianza}]
  Call $\mathcal L$ the Markov generator, then we have to check $[\pi\mathcal L](\eta)=0$ for every configuration $\eta\in \OLN$.
This can be rewritten as
\begin{eqnarray}
  \label{eq:8}
  \sum_{\sigma\ne \eta}\pi(\sigma)\mathcal
  L(\sigma,\eta)+\pi(\eta)\mathcal L(\eta,\eta)=0.
\end{eqnarray}
Since the generator has sum zero on rows, this is equivalent to
\begin{eqnarray}
  \label{eq:9}
  \pi(\eta)\left[
\sum_{\sigma\ne\eta}\frac{\pi(\sigma)}{\pi(\eta)}\mathcal
  L(\sigma,\eta)-\sum_{\sigma\ne\eta}\mathcal L(\eta,\sigma)
\right]=0.
\end{eqnarray}
Now it is obvious from the definition of the dynamics that
\begin{eqnarray}
  \label{eq:10}
  \sum_{\sigma\ne\eta}\mathcal L(\eta,\sigma)=\sum_p a_{r(p)}\, \frac{(1-q^{B_p})(1-q^{D_p+1})}{(1-q^{C_p+1})}=:S_1(\eta).
\end{eqnarray}
On the other hand, we claim (see proof below) that
\begin{eqnarray}
  \label{eq:11}
  \sum_{\sigma\ne\eta}\frac{\pi(\sigma)}{\pi(\eta)}\mathcal
  L(\sigma,\eta)=\sum_p a_{r(p)+1}\frac{(1-q^{A_p+1})(1-q^{E_p})}{(1-q^{F_p+1})}=:S_2(\eta),
\end{eqnarray}
where $r(p)+1$ equals $1$ if $r(p)=N$. Finally, we claim that
\begin{eqnarray}
  \label{eq:12}
  S_1(\eta)=S_2(\eta)\quad \text{for every} \quad \eta\in \Omm.
\end{eqnarray}

\noindent\emph{Proof of \eqref{eq:11}}
An update $\sigma\to\eta$ means that a family $V^+_{p}$ in $\sigma$ has been shifted by $e_1$. Call, as in Definition \ref{def:famiglie}, $r^+_p$ the highest particle in $V^+_{p}$. The reverse move consists in  shifting $V^-_{r^+_p}$ by $-e_1$. Call $\eta_p$ the configuration obtained by $\eta$ by shifting $V^-_p$ by $-e_1$. We see then that the configurations $\eta_p $, with $p$ running over all particles, exhaust the configurations $\sigma$ contributing to \eqref{eq:11}.
% Recall Definition \ref{def:famiglie} of $B^-_p$ and let $p^-$ the
% ``lowest'' particle in  $B^-_p$. %   Particles can be sub-divived into families; each
% family contains $\le N-1$ particles, each in a different row, all in
% neighboring rows, and such that the one in the lowest row if it moves
% to the right pushes the rest of the family, the second from the bottom
% would push the family except the lowest one, etc. Definition to be
% written more formally. There is no pushing among different
% families. Particles that do not push and are not pushed by anyone form
% trivial families. We label families with an index $k$.
% Take a particle $p$,
% shift by $1$ to the left the particles in $B^-_p$ and call $\sigma_p$
% the configuration thus obtained.
% For each family $k$, one
% can move one step left either the lowest particle, or the two lowest,
% etc, or the whole family. The configuration thus obtained is always a
% legal one (I would say, by definition of families). Call
% $\sigma_{k,\ell}$ the configuration obtained starting from
% $\eta$ by shifting by $1$ to the
% left the lowest $\ell$ particles in family $k$.
Next we claim that
\begin{eqnarray}
  \label{eq:15}
  \frac{\pi(\eta_p)}{\pi(\eta)}\mathcal
  L(\eta_{p},\eta)= a_{r(p)+1} \frac{(1-q^{A_p+1})(1-q^{E_p})}{(1-q^{F_p+1})}.
\end{eqnarray}
If this is true, then \eqref{eq:11} is proven. Label $p_{-\ell+1},\equiv r^-_p,p_{-\ell+2},\dots,p_0\equiv p$ the $\ell=|V^-_p|\le N-1 $ particles that are
shifted left in the move $\eta\to\eta_p$, starting from the lowest one.
We have from \eqref{eq:6},
\begin{eqnarray}
  \label{eq:16}
 \mathcal L(\eta_p,\eta)=a_{r(p_{-\ell+1})} \frac{(1-q^{B_{p_{-\ell+1}}+1})(1-q^{D_{p_{-\ell+1}}})}{1-q^{C_{p_{-\ell+1}}}}.
\end{eqnarray}
As for the ratio of probabilities, one sees that
\begin{align}
  \label{eq:17}
   \frac{\pi(\eta_p)}{\pi(\eta)}&= \frac{a_{r(p_0)+1}}{a_{r(p_{-\ell+1})}} \frac{(q;q)_{F_{p_0}}(q;q)_{C_{p_{-\ell+1}}}}{(q;q)_{F_{p_0}+1}(q;q)_{C_{p_{-\ell+1}}-1}}\\
  \nonumber&\,\,\,\,\, \times\prod_{j=-\ell+1}^{0}\frac{(q;q)_{A_{p_j}+1}(q;q)_{D_{p_j}-1}}{(q;q)_{A_{p_j}}(q;q)_{D_{p_j}}}\frac{(q;q)_{E_{p_j}}(q;q)_{B_{p_j}}}{(q;q)_{E_{p_j}-1}(q;q)_{B_{p_j}+1}}\\
  \nonumber&=\frac{a_{r(p_0)+1}}{a_{r(p_{-\ell+1})}} \frac{(1-q^{C_{p_{-\ell+1}}})}{(1-q^{F_{p_0}+1})}\prod_{j=-\ell+1}^{0}\frac{(1-q^{A_{p_j}+1})(1-q^{E_{p_j}})}{(1-q^{D_{p_j}})(1-q^{B_{p_j}+1})}.
\end{align}
Note that the $F$'s and $C$'s of intermediate particles do not appear
because they are exactly zero, both in $\eta$ and in $\eta_p$.
Also, note that
$B_{p_0}=A_{p_1},\dots, B_{p_{-\ell+2}}=A_{p_{-\ell+1}}$ and
$E_{p_{-1}}=D_{p_0},\dots,E_{p_{-\ell+1}}=D_{p_{-\ell+2}}$.
Therefore, \eqref{eq:17} simplifies into
\begin{gather}
  \label{eq:18}
 \frac{ a_{r(p_0)+1}}{a_{r(p_{-\ell+1})}}\,\frac{(1-q^{C_{p_{-\ell+1}}}) (1-q^{A_{p_0}+1})(1-q^{E_{p_0}})}{(1-q^{B_{p_{-\ell+1}}+1})(1-q^{D_{p_{-\ell+1}}}) (1-q^{F_{p_0}+1})}.
\end{gather}
Once we multiply \eqref{eq:16} by \eqref{eq:18} and we recall that
$p_0=p$, we obtain \eqref{eq:15} and therefore
\eqref{eq:11}.
\end{proof}

\begin{proof}[Proof of \eqref{eq:12}]
It is sufficient to prove that $S_1-S_2$ is independent of $\eta$. In this case, it must be zero because $\sum_\eta[ \pi\mathcal L](\eta)=0$ by conservation of the probability.
Thanks to Lemma \ref{lemma1} it is sufficient to show that $S_1-S_2$ is unchanged when a single particle is moved by $e_1$. The sums in $S_1,S_2$ make sense also when horizontal particle positions are real numbers; given a particle $p$ that can be moved by $e_1$, we let $\eta(s)$ be the configuration where $p$ is moved by $se_1,s\in [0,1]$ and we prove that $\partial_s (S_1-S_2)=0$. %  Recall that $p_1,\dots,p_6$ are the six neighbors of $p$, labelled
% clockwise starting from the one to its right.
Note that not only the values of $A_p,\dots,F_p$ depend on $s$, but also  $D_{p_1},E_{p_2},F_{p_3},A_{p_4},B_{p_5},C_{p_6}$ do. We find by direct computation
\begin{gather}
   \frac{\partial_s S_1}{\log q}= a_{r(p)} S_{10} + a_{r(p)+1}S_{11}, \quad \text{and}\quad
  \frac{\partial_s S_2}{\log q}= a_{r(p)} S_{20} + a_{r(p)+1}S_{21},
\end{gather}
where
\begin{align}
S_{10} &= q^{B_p}\frac{1-q^{D_p+1}}{1-q^{C_p+1}}-q^{D_p+1}\frac{1-q^{B_p}}{1-q^{C_p+1}}\\
\nonumber &\,\,\,\,\,+q^{C_p+1}\frac{(1-q^{B_p})(1-q^{D_p+1})}{(1-q^{C_p+1})^2}+q^{D_{p_1}+1}\frac{1-q^{B_{p_1}}}{1-q^{C_{p_1}+1}},\\
S_{11} &= -q^{C_{p_6}+1}\frac{(1-q^{D_{p_6}+1})(1-q^{B_{p_6}})}{(1-q^{C_{p_6}+1})^2}-q^{B_{p_5}}\frac{1-q^{D_{p_5}+1}}{1-q^{C_{p_5}+1}},\\
S_{20} &= q^{F_{p_3}+1}\frac{(1-q^{A_{p_3}+1})(1-q^{E_{p_3}})}{(1-q^{F_{p_3}+1})^2}+q^{E_{p_2}}\frac{1-q^{A_{p_2}+1}}{1-q^{F_{p_2}+1}},\\
S_{21} &= q^{A_p+1}\frac{1-q^{E_p}}{1-q^{F_p+1}}-q^{E_p}\frac{1-q^{A_p+1}}{1-q^{F_p+1}}\\
\nonumber &\,\,\,\,\,-q^{F_p+1}\frac{(1-q^{A_p+1})(1-q^{E_p})}{(1-q^{F_p+1})^2}-q^{A_{p_4}+1}\frac{1-q^{E_{p_4}}}{1-q^{F_{p_4}+1}}.
\end{align}
Using $D_{p_1}=A_p,
C_{p_1}=A_p-B_b,C_{p_6}=F_p,D_{p_6}=E_p+F_p,B_{p_6}=A_p-F_p,B_{p_5}=E_p,C_{p_5}=D_p-E_p$
we get
\begin{align}
  \label{eq:14bis}
S_{10} &= q^{B_p}\frac{1-q^{D_p+1}}{1-q^{C_p+1}}-q^{D_p+1}\frac{1-q^{B_p}}{1-q^{C_p+1}}\\
&\nonumber\,\,\,\,\,+q^{C_p+1}\frac{(1-q^{B_p})(1-q^{D_p+1})}{(1-q^{C_p+1})^2}
+q^{A_{p}+1}\frac{1-q^{B_{p_1}}}{1-q^{A_p-B_p+1}}\\
S_{11}&=-q^{F_{p}+1}\frac{(1-q^{F_p+E_p+1})(1-q^{A_p-F_p})}{(1-q^{F_{p}+1})^2}-q^{E_p}\frac{1-q^{D_{p_5}+1}}{1-q^{D_p-E_p+1}}
\end{align}
while using $A_{p_4}=D_p,E_{p_4}=E_p+D_{p_5}-D_p,F_{p_4}=D_p-E_p,F_{p_3}=C_p,A_{p_3}=B_p+C_p,E_{p_3}=D_p-C_p,A_{p_2}=A_p+B_{p_1}-B_p,E_{p_2}=B_p,F_{p_2}=A_p-B_p, $
\begin{align}
  \label{eq:13bis}
S_{20}&=q^{C_p+1}\frac{(1-q^{B_p+C_p+1})(1-q^{D_p-C_p})}{(1-q^{C_p+1})^2}+q^{B_p}\frac{1-q^{A_p+B_{p_1}-B_p+1}}{1-q^{A_p-B_p+1}}\\
S_{21}&=q^{A_p+1}\frac{1-q^{E_p}}{1-q^{F_p+1}}-q^{E_p}\frac{1-q^{A_p+1}}{1-q^{F_p+1}}\\
&\nonumber\,\,\,\,\,-q^{F_p+1}\frac{(1-q^{A_p+1})(1-q^{E_p})}{(1-q^{F_p+1})^2}-q^{D_{p}+1}\frac{1-q^{E_p+D_{p_5}-D_p}}{1-q^{D_p-E_p+1}}
\end{align}
Note that, when taking the difference $S_{10}-S_{20}$, the dependence
on $B_{p_1}$ cancels and likewise for the difference $S_{11}-S_{21}$
the dependence on $D_{p_5}$ cancels. Finally, one checks that each of
these differences, which a priori depends on $q,A_p,\dots,F_p$, is actually identically zero and \eqref{eq:12} is proven. This concludes the proof of Theorem \ref{th:invarianza}.
\end{proof}

% \section{Comments/questions}

% Assuming this is correct, maybe one would like to work also in
% infinite volume

% \begin{itemize}
% \item is the classifications of Gibbs measure know?
% \item for your process with initial packed configuration, can you give
%   the particle correlations at equal times, and maybe asymptotics for
%   large times in the bulk? this should give the Gibbs measures.
% \item in my work on B-F dynamics, to extend to infinite lattice I use:
%   1) knowledge of what infinite volume Gibbs measure are, in
%   particular 2) the fact that inter-particle gaps have very light
%   tails (something like decaying faster than exponential). To get that
%   height fluctuations grow less than sqrt log t, I also use 3) that
%   height fluctuations in the Gibbs measure are very small (logarithmic
%   in space, like GFF).
% \end{itemize}

\section*{Acknowledgments}

I. C. was partially supported by the NSF DMS-1208998, by a Clay Research Fellowship, by the Poincar\'{e} Chair, and by a Packard Fellowship for Science and Engineering. F. T. was partially funded by Marie Curie IEF Action “DMCP- Dimers, Markov chains and Critical Phenomena”, grant agreement n. 621894. This work was initiated during the Limit Shapes conference at ICERM and was further discussed at the Statistical Mechanics, Integrability and Combinatorics program at GGI and during Minerva Foundation lectures at Columbia University, delivered by F. T. We appreciate the hospitality and support of these institutes.


\begin{thebibliography}{99}
\bibitem{Andjel}
E.~Andjel, \emph{Invariant measures for the zero range process} Ann. Probab. {\bf 10} (1982), 525--547.
\bibitem{BetaPolymer}
G. Barraquand, I. Corwin, \emph{Random-walk in Beta-distributed random environment}, arXiv:1503.04117.
\bibitem{BorodinR}
A. Borodin, \emph{On a family of symmetric rational functions}, arXiv:1410.0976.
\bibitem{BorodinBufetov}
A. Borodin, A. Bufetov, \emph{An irreversible local Markov chain that preserves the six vertex model on a torus}, in preparation.
\bibitem{BF1} A. Borodin, P. L. Ferrari, \emph{Anisotropic KPZ growth in $2 + 1$ dimensions}, Comm. Math. Phys. {\bf  325} (2014), 603--684.
\bibitem{BF2} A. Borodin, P. L. Ferrari, \emph{Anisotropic KPZ growth in $2+1$ dimensions: fluctuations and covariance structure}, J. Stat. Mech. (2009), P02009.
\bibitem{BC} A. Borodin, I. Corwin, \emph{Macdonald processes},
  Probab. Theory Relat. Fields {\bf 158} (2014), 225--400.
\bibitem{BorPet} A. Borodin, L. Petrov, \emph{Nearest neighbor Markov dynamics on Macdonald processes}, arXiv:1305.5501.
\bibitem{EvansMajumdarZai}
M. Evans, S. Majumdar, R. Zia, \emph{Construction of the factorized steady state distribution in modesl of mass transport}, J. Stat. Mech. (2004).
\bibitem{GatesWestcott}
D. J. Gates, M. Westcott, \emph{Stationary states of crystal growth in three dimensions}, J. Stat. Phys. {\bf 81} (1995), 681--715.
% \bibitem{GoncalvesLandimCTonninelli}
% P. Goncalves, C. Landim, C. Toninelli, \emph{Hydrodynamic limit for a particle system with degenerate rates}, Ann. Inst. Henri Poincare Probab. Stat. {\bf 45} (2009), 887--909.
\bibitem{Kenyon} R. Kenyon, \emph{Lectures on Dimers}, arXiv:0910.3129.
\bibitem{Kenyon-Okounkov-Sheffield} R. Kenyon, A. Okounkov, S. Sheffield, \emph{Dimers and amoebae}, Ann. Math. {\bf 163} (2006), 1019--1056.
\bibitem{OConnellYor}
N. O'Connell, M. Yor, \emph{Brownian analogues of Burke's theorem}, Stoch. Proc. Appl. {\bf 96} (2001), 285--304.
\bibitem{PrahoferSpohn} M. Pr\"{a}hofer, H. Spohn, \emph{An exactly solved model of three dimensional surface growth in the anisotropic KPZ regime}, J. Stat. Phys. {\bf 88} (1997), 999--1012.
\bibitem{QuastelValko}
J. Quastel, B. Valko, \emph{Diffusivity of lattice gases}, Arch. Rat. Mech. Anal. {\bf 210} (2013), 269--320.
\bibitem{SasamotoSpohnSuper}
T. Sasamoto, H. Spohn, \emph{Superdiffusivity of the 1D lattice Kardar-Parisi-Zhang equation}, J. Stat. Phys. {\bf 137} (2009), 917--935.
\bibitem{SeppLogGamma}
T. Sepp\"{a}l\"{a}inen, \emph{Scaling for a one-dimensional directed polymer with boundary conditions}, Ann. Probab. {\bf 40} (2012), 19--73.
\bibitem{Spitzer}
F. Spitzer, \emph{Interaction of Markov processes}, Adv. Math. {\bf 5} (1970), 246--290.
\bibitem{Toninelli} F. L. Toninelli, \emph{A $(2+1)$-dimensional growth process with explicit stationary measures}, arXiv:1503.05339.
\end{thebibliography}
\end{document}